\documentclass[%
 reprint,
 ,amsmath
 ,amssymb,
 aps,
 onecolumn,
]{revtex4-2}
\usepackage{graphicx}% Include figure files
\usepackage{dcolumn}% Align table columns on decimal point
\usepackage{bm}% bold math

\newcommand{\Tr}{\mathop{ \rm tr}}
\newcommand{\tr}{\mathop{ \rm tr}}
\newcommand{\sign}{\mathop{ \rm sign}}
\newenvironment{proof}{\par\noindent{\bf Proof\ }}{}
\newtheorem{theorem}{Theorem}
\newtheorem{proposition}{Proposition}

\newtheorem{definition}{Definition}

\def \H{{\mathbb H}}
\def \X{{\mathcal X}}

\def \cF{{\mathcal F}}
\def \cK{{\mathcal K}}
\def \cX{{\mathcal X}}
\def \cP{{\mathcal P}}

\def \bR{{\mathbb R}}
\def \bZ{{\mathbb Z}}
\def \ds{\displaystyle}
\DeclareSymbolFont{usualmathcal}{OMS}{cmsy}{m}{n}
\DeclareSymbolFontAlphabet{\mathcal}{usualmathcal}

\newcommand{\norm}[1]{\left\lVert#1\right\rVert}
\DeclareMathOperator{\EX}{\mathbb{E}}% expected value
\newtheorem{prop}{Proposition}
\begin{document}

\preprint{APS/123-QED}

\title{Theory and applications of the Sum-Of-Squares technique}% Force line breaks with \\

\author{Francis Bach}
 \affiliation{Inria, Ecole Normale Supérieure
PSL Research University, Paris, France}%Lines break automatically or can be forced with \\
\author{Elisabetta Cornacchia}%
\affiliation{École Polytechnique Fédérale de Lausanne (EPFL), Mathematical Data Science laboratory, CH-1015 Lausanne, Switzerland
}%

\author{Luca Pesce}
\affiliation{École Polytechnique Fédérale de Lausanne (EPFL), Information, Learning and Physics laboratory, CH-1015 Lausanne, Switzerland
}%
\author{Giovanni Piccioli}
\affiliation{%
École Polytechnique Fédérale de Lausanne (EPFL), Statistical
Physics of Computation laboratory, CH-1015 Lausanne, Switzerland}%
 \email{giovanni.piccioli@epfl.ch}



\begin{abstract}
The Sum-of-Squares (SOS) approximation method is a technique used in optimization problems to derive lower bounds on the optimal value of an objective function. By representing the objective function as a sum of squares in a feature space, the SOS method transforms non-convex global optimization problems into solvable semidefinite programs. This note presents an overview of the SOS method. We start with its application in finite-dimensional feature spaces and, subsequently, we extend it to infinite-dimensional feature spaces using reproducing kernels (k-SOS). Additionally, we highlight the utilization of SOS for estimating some relevant quantities in information theory, including the log-partition function.
% The Sum of Squares (SOS) approximation method is a technique for finding a lower bound to the optimal value of an objective function in a minimization problem. By approximating the function with a sum of squares in a feature space, the original non-convex global optimization problem is transformed into an efficiently solvable semidefinite program. In this framework, we first present finite-dimensional feature space approximations, then extend it to infinite-dimensional feature spaces using kernels (k-SOS). Finally, using SOS, we consider the estimation of various relevant quantities in information theory and reinforcement learning.
\end{abstract}

\maketitle

%\tableofcontents
\section{Lecture 1}
\subsection{All problems are convex!}
Let us consider a global optimization problem where we want to minimize a function $h: \X \to \mathbb{R}$. At first we only look for the infimum $h(\hat{x})$, and not where it is attained. Moreover, we are interested in considering general assumptions on $h(x)$, in particular we do not want to put constraints  on the convexity of the objective. However, a very important fact we exploit in the following, is that one can always rephrase a non-convex problem into a convex one (see Fig.~\ref{fig:apac} for intuition):
\begin{figure}[t]
\centering
     \includegraphics[width= 0.6 \textwidth]{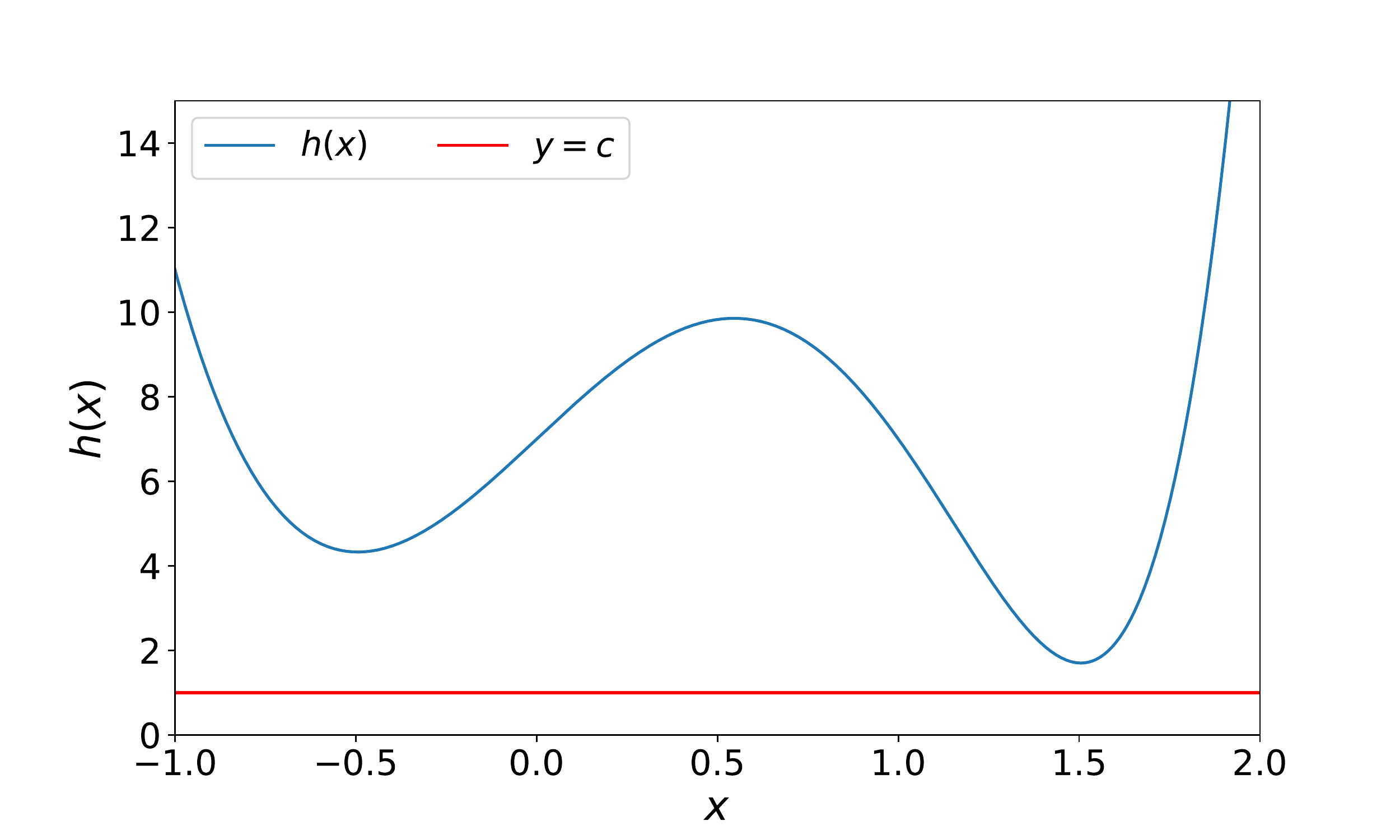}
     
      \vspace*{-.2cm}
      
      \caption{All problems are convex. In order to find the global minimum of $h(x)$ one needs to find the greatest $c$ which satisfy $h(x) - c \geq 0, \,\, \forall x \in \X$.}
\label{fig:apac}
\end{figure}
\begin{align}
\inf_{x \in \X} h(x)  = \sup_{c \in \mathbb{R}} \ c  \,\,\mbox{ such that } \forall x \in \X, \,\, h(x) - c \geqslant 0.  \,\,\,
\label{eq:glob_opt}
\end{align}

Indeed the right hand side is clearly a convex problem, i.e., minimization over a convex set plus linear constraints in $c$, while on the left hand side we have a non-convex one in the most general setting. The catch is that we are imposing constraints over a dense set (if we do not make any assumption on $\X$)! After this mapping, the main technical challenge becomes dealing with non-negativity constraints over the set $\X$, i.e., $h(x)-c \geq 0,  \, \forall x \in \X$. A crucial point that we must learn from the equivalence between the two problems above is that we need a computationally efficient (essentially linear) way to manipulate non-negative functions.

Convex duality leads to another natural formulation:
\begin{align}
\inf_{x \in \X} h(x) & = \sup_{c \in \mathbb{R}} \ c \mbox{ such that } \forall x \in \X, h(x) - c \geqslant 0 \\
& =  \sup_{c \in \mathbb{R}} \inf_{ \mathcal{P}^{< \infty}_{+}(\mathcal{X})}  \left \{ c + \int_\X (  h(x) - c) d\mu(x) \right \} \label{eq:step2}
\\
& =   \inf_{\mu  \in \mathcal{P}^{< \infty}_{+}(\mathcal{X}) } \sup_{c \in \mathbb{R}} \left \{ c +  \int_\X (  h(x) -c ) d\mu(x) \right \} \label{eq:step3}\\
& =   \inf_{ \mu  \in \mathcal{P}^{< \infty}_{+}(\mathcal{X})}  \left\{ \int_\X h(x) d\mu(x) \right \} \mbox{ such that } \int_\X d\mu(x) = 1,
\label{eq:step4}
\end{align}
\noindent where we denoted the space of finite positive measures on $\mathcal{X}$ as $\mathcal{P}^{< \infty}_{+}(\mathcal{X})$. We rewrite the global optimization problem in eq.~\eqref{eq:glob_opt} thanks to the introduction of the Lagrangian in eq.~\eqref{eq:step2}; due to the positiveness of the constraint the measure must be positive and we take the infimum over $\mathcal{P}^{< \infty}_{+}(\mathcal{X})$. By exploiting the convexity of the problem (strong duality)
we can invert $\inf$ and $\sup$ in eq.~\eqref{eq:step2} to obtain the dual formulation of the optimization problem in eq.~\eqref{eq:step3}. The interpretation of the measure for the optimization problem is given in Fig.~\ref{fig:measure}. By solving the dual problem we can access both the minimum $h(\hat{x})$ and the minimizer $\hat{x}$ thanks to the knowledge of the measure attaining the infimum in eq.~\eqref{eq:step4}. %\bl{GP: I think ~\eqref{eq:step3} is the dual problem and that the equality with ~\eqref{eq:step2} holds because of convexity (strong duality).}\textcolor{orange}{LP: Yes I think you are right, I thought the dual was the lagrangian formulation. Better now?}
\begin{figure}
\centering
     \includegraphics[width= 0.6 \textwidth]{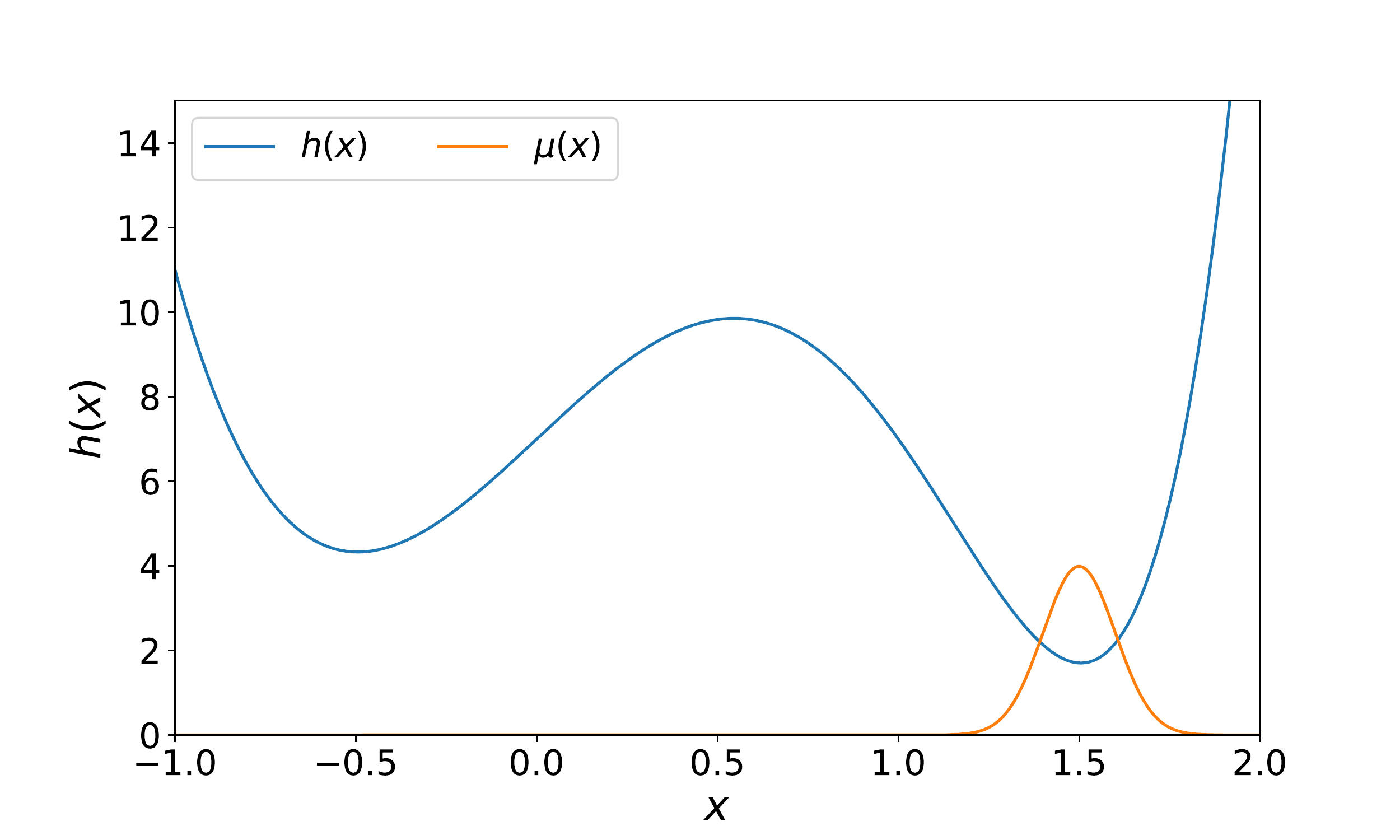}
    \vspace*{-.2cm}
    
      \caption{Cartoon plot of the positive measure $\mu(x)$ which peaks around the minimizer $\hat{x}$. The measure $\mu(x)$ that attains the $\inf$ in eq.~\eqref{eq:step4} is $\delta_{\hat x}$.
}
\label{fig:measure}
\end{figure}

\subsection{Sum-of-squares representation of non-negative functions}

We need a computationally efficient (essentially linear) way to manipulate non-negative functions. In this section we introduce one idea that goes in this direction, originally introduced by \cite{Lasserre_original,parrilo2000structured}: represent non-negative functions as ``sums of squares'' (SOS).
Let us consider a feature map $\varphi: \X \to \mathbb{C}^d$, the goal is to represent functions as quadratic forms:
\begin{align}
    h(x) = \varphi(x)^\ast H \varphi(x),
    \label{eq:quad_form}
\end{align}
\noindent where $H \in \H_d$, i.e., set of Hermitian matrices in $\mathbb{C}^{d \times d}$. Remember that we want to deal with a real objective function $h$ and a sufficient condition in order to achieve that is to assume $H$ to be Hermitian.
Many functions admit a representation of this form:
\begin{itemize}
    \item Polynomials: $\varphi(x) = (1,x,x^2, \dots )$.
    \item Trigonometric polynomials: $\varphi(x) = (e^{i \pi x},e^{-i \pi x}, e^{i 2 \pi x}, e^{-i 2 \pi x}, \dots )$.
\end{itemize}
It is important to stress that the representation of $h$ associated with $H$ is not unique. Consider the linear span of the set $\{ \varphi (x)\varphi(x)^\ast\,:\, x\in \X\}$, and denote it as $\mathcal{V}$.
If in eq.~\eqref{eq:quad_form} we substitute $H$ with $H + H^{\prime}$, where $H^{\prime} \in \mathcal{V}^\perp$\footnote{The orthogonality is with respect to the following dot product in matrix space $\langle A,B\rangle=\tr[A^\ast B]$.}, the left hand side would be left unchanged.

A key assumption we make is that the constant function can be represented through a certain (non-unique) $U \in \H_d$:
\begin{align}
     \varphi(x)^\ast U \varphi(x) = 1, \,\,\, \forall x \in \X.
     \label{eq:repr_const}
\end{align}

We call a function a Sum of Squares (SOS) if it can be expanded in the following form:
\begin{align}
    h(x) = \sum_{i \in \mathcal{I}} g_i(x)^2,
\end{align}
\noindent where we introduced the generic set of functions $\{g_i: \X \to \mathbb{R} \}_{i \in \mathcal{I}}$.  

We can connect this definition with the representation as a quadratic form in eq.~\eqref{eq:quad_form} by characterizing the SOS functions using the following proposition: 
\begin{prop}
The objective function $h$, represented as $h(x) = \varphi(x)^\ast H \varphi(x)$, is a SOS if and only if $H \succcurlyeq 0$ and $H \in \H_d$.
\end{prop}
The statement is easily understood by doing the spectral decomposition of the matrix $H$. Indeed let $H = \sum_{i \in I} \lambda_i u_i u_i^\ast$, then we can massage eq.~\eqref{eq:quad_form} to obtain:
\begin{align}
h(x) = \sum_{i \in I } \big[ \sqrt{\lambda_i} u_i^\ast \varphi(x) \big]^2,
\end{align}
\noindent which is a ``sum-of-squares'' (SOS). It is interesting to note that the number of squares in the SOS decomposition will be equal to the rank of the matrix $H$. 

A rather simple statement which is pivotal for the following discussion is:
\begin{prop}
    If the objective function $h$ is a SOS, then $h$ is non-negative.
    \label{prop:1}
\end{prop}
% \bl{(GP: but this is not the definition of SOS)} \textcolor{orange}{LP: I don't get the point, isn't this the definition we gave? It is indeed equivalent to a SOS in the basis, no?}\bl{GP: yes, it is equivalent but the flow I was used to was defining $f$ to be SOS if $f=\sum_i g_i^2$ and then say that every function of the form $\varphi H\varphi$ can be written in SOS form. But yeah, in the end it's the same thing, so feel free to ignore me. }

 The converse of prop.~\ref{prop:1} is not true and not all non-negative functions are SOS! We focus our attention in the following on understanding when the characterization is tight. 
 
However an important point is that, if we are given a representation of a function $h$ as a quadratic form in eq.~\eqref{eq:quad_form}, checking that $h$ is a SOS is a convex feasibility problem:
\begin{prop}
    $h(x)=\varphi(x)^\ast H \varphi(x)$ is a SOS if and only if it exists $H^{\prime} \in \mathcal{V}^\perp$ such that $H-H^{\prime} \succcurlyeq 0$.
    \label{prop:2}
\end{prop}
Indeed if the space $\mathcal{V}$ is known, and for most applications we will study it will be, the problem in prop.~\ref{prop:2} is a well-defined SDP. 

From the discussion above we understand that the linear span $\mathcal{V}$ plays a key role in the SOS contruction. We can analyze its form in the simple examples we already introduced:
\begin{itemize}
    \item Polynomials:  The space $\mathcal{V}$ is the set of Hankel matrices, indeed  $\left[\varphi(x)\varphi(x)^{\ast}\right]_{i,j} = x^{(i+j)}$.
    \item Trigonometric polynomials: The space $\mathcal{V}$ is the set of Toeplitz matrices, indeed  $\left[\varphi(x)\varphi(x)^{\ast}\right]_{\omega,\omega'} = e^{i \pi (\omega - \omega')x} $.
\end{itemize}

Using the propositions we introduced above we can reformulate the optimization problem as:
\begin{align}
\inf_{x \in \X} h(x) & = \sup_{c \in \mathbb{R}} \ c \mbox{ such that } \forall x \in \X, h(x) - c \geqslant 0  \\
\label{eq:sos_primal_relaxed}
& \geqslant  \sup_{c \in \mathbb{R}, \ A \succcurlyeq 0} \ c\  \mbox{ such that } \forall x \in \X,  h(x) - c = \varphi(x)^\ast A \varphi(x) \\
& =   \sup_{c \in \mathbb{R}, \ A \succcurlyeq 0} \ c \ \mbox{ such that } H - c U   - A \in \mathcal{V}^\perp
\label{eq:sos_view}
,\end{align}
\noindent where we used Prop.~\ref{prop:1} knowing that the SOS functions are included in the set of non-negative functions (hence the relaxation in the cost), and that the constant function can be represented by $U$. Finally we assumed that $h(x)=\varphi(x)H\varphi(x)^{\ast}$ (i.e.,  $h$ is SOS). The problem above is a \textit{semidefinite program}! Therefore we can solve it numerically on a computer.

The SOS relaxation modifies as well the dual problem:
\begin{align}
\inf_{x \in \X} h(x) & = \sup_{c \in \mathbb{R}} \ c \mbox{ such that } \forall x \in \X, h(x) - c \geqslant 0  \\
 &\geqslant  \sup_{c \in \mathbb{R}, \ A \succcurlyeq 0} \ c \ \mbox{ such that } \forall x \in \X,  h(x) - c = \varphi(x)^\ast A \varphi(x)
 \label{eq:step_dual1}\\
& = \sup_{c \in \mathbb{R}, \ A \succcurlyeq 0}  \inf_{ \mu \in \mathcal{P}^{<\infty} (\mathcal{X})}    c+\int_\X ( h(x) -c - \varphi(x)^\ast A \varphi(x) ) d\mu(x) \label{eq:step_dual2}
\\
& =   \inf_{  \mu \in \mathcal{P}^{<\infty} (\mathcal{X})} \sup_{c \in \mathbb{R}, \ A \succcurlyeq 0}   c+\int_\X ( h(x) -c - \varphi(x)^\ast A \varphi(x) ) d\mu(x) \label{eq:step_dual3} \\
& =   \inf_{  \mu \in \mathcal{P}^{<\infty} (\mathcal{X})}   \int_\X h(x) d\mu(x) \mbox{ such that } \int_\X d\mu(x) = 1 \mbox{ and }
\int_\X \varphi(x)\varphi(x)^\ast d\mu(x) \succcurlyeq 0.
\label{eq:moment_sos}
\end{align}

We write the Lagrangian in eq.~\eqref{eq:step_dual2}, and we exploit strong duality in eq.~\eqref{eq:step_dual3} to exchange inf and sup. While the constraint imposed by $c$ yields again simply the normalization of the measure as in eq.~\eqref{eq:step4}, a maximization over $A$ is more elaborated. 
The optimization concerning $A$ is $\sup_{ A \succcurlyeq 0} -\int_\X \varphi(x)^\ast A \varphi(x)  d\mu(x) =-\inf_{A\succcurlyeq 0}  \int_\X \varphi(x)^\ast A \varphi(x)  d\mu(x)= -\inf_{A\succcurlyeq 0} \tr[AB]$, with $B=\int_\X \varphi(x)\varphi(x)^\ast$. The solution is
\begin{equation}
    -\inf_{A\succcurlyeq 0} \tr[AB]=\begin{cases} 0 &\text{ if  $B\succcurlyeq 0$}\\
+\infty &\text{ otherwise }
\end{cases}
\end{equation}
in fact if $B\succcurlyeq 0$ then $\tr[AB]\geq 0$ for all $A\succcurlyeq 0$, hence a minimizer is $A=0$. Suppose instead that $B$ has a negative eigenvalue with eigenvector $v_-$, then picking $A=\lambda v_-v_-^T$ one can make $\tr[AB]$ arbitrarily negative by increasing $\lambda$.
% \bl{GP: qui ho cambiato la tua spiegazione perché penso non funzionasse benissimo (veniva il vincolo $\int_\X \varphi(x)^\ast \varphi(x)  d\mu(x)=0$) e penso di aver trovato quella giusta. Controlla tu se va bene.}
%exploiting that every scalar is equal to its trace we rewrite $\int_\X  \varphi(x)^\ast A \varphi(x) d\mu(x)$ as $\tr \left[\int_\X  \varphi(x)^\ast A \varphi(x) d\mu(x)\right]$, then we use the invariance under cyclic permutation of the trace to rearrange it as $\tr\left[A\int_\X (\varphi(x)\varphi(x)^\ast ) d\mu(x)\right]$, which can be easily differentiated with respect to $A$. The associated constraint tells us that we do not need to require anymore $\mu$ to be a finite positive measure on $\X$ as in eq.~\eqref{eq:step4}, instead we only need to have positive moments, i.e. positiveness of the covariance matrix 

Assuming the objective admits a representation as in  eq.~\eqref{eq:quad_form} $h(x) = \varphi(x)^\ast H \varphi(x)$, a simple reformulation of the problem in eq.~\eqref{eq:glob_opt} is:
\begin{align}
\inf_{x \in \X} h(x) &= \inf_{x \in \X}
\tr \Big[ H  \varphi(x)\varphi(x)^\ast \Big]
 \\
& =  \inf_{ \Sigma \in \mathcal{K} }\  \tr [ \Sigma H ],
\label{eq:convex_hull}
\end{align}
\noindent where we exploited the linearity of the problem in order to optimize over $\mathcal{K}$, defined as the closure of the convex hull of $\varphi(x)\varphi(x)^\ast$, indeed whenever we maximize linear objective over a convex hull we end up in one extremal point. An illustration of the basic geometry involved is given in Fig.~\ref{fig:hull}. By looking at eq.~\eqref{eq:convex_hull} we see that the relaxation in eq.~\eqref{eq:moment_sos} is equivalent to find upper-bounds of $\mathcal{K}$ as
$\widehat{\mathcal{K}} =
\big\{  \Sigma \in \mathcal{V}, \ \tr[ U\Sigma ] =1, \ \Sigma \succcurlyeq 0 \big\}$, with $\Sigma=\int_\X \varphi(x)\varphi(x)^\ast d\mu(x)$. Notice that $\mathcal{K}\subseteq\widehat{\mathcal{K}}$. Indeed, the final goal of optimization is to find tractable outer approximations of $\mathcal{K}$.

\begin{figure}
\centering
     \includegraphics[width= 0.6 \textwidth]{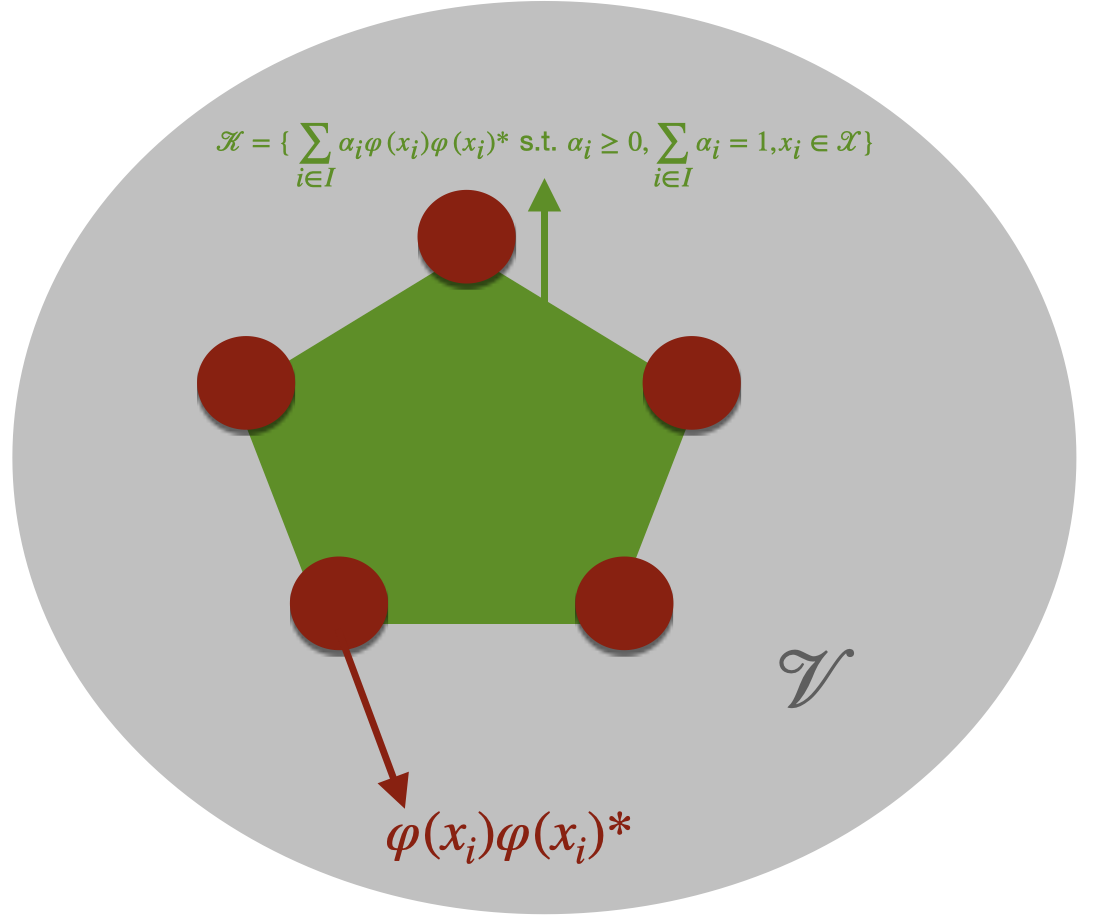}
      \caption{Cartoon plot of the convex hull in the linear span $\mathcal{V}$. }
\label{fig:hull}
\end{figure}

\subsection{Tightness of the approximation}
We saw in previous section that we can approach a SOS relaxation from two point of views:
\begin{itemize}
    \item \textit{Primal point of view}: Relax the non-negative constraint to be a SOS (eq.~\eqref{eq:sos_view}).
    \item \textit{Dual (moment) point of view}: Relax the strict assumption on the positive sign of the measure $\mu$, require only positiveness of the moments. It is equivalent to replace the convex hull $\mathcal{K}$ with the affine hull of $\varphi(x)\varphi(x)^\ast$ intersected with the SDP cone which we called $\hat{\mathcal{K}}$ (eq.~\eqref{eq:moment_sos}).
\end{itemize}
The two characterizations are completely equivalent. If $\mathcal{K} = \hat{\mathcal{K}}$ then the non-negative constraint can be expressed as a SOS, and viceversa. A natural question we investigate is: when is the SOS approximation tight? We analyze in the following a series of (simple) example in which we can give detail on the tightness of the approximation:
\begin{itemize}
\item
\textbf{Finite set with injective embedding:}
Let $\X$ be a finite set $\X = \{1,2,\dots,n\}$. Consider a \textit{one-hot encoding} mapping $\varphi(i) = e_i$, with $e_i$ the $i$-th element of the canonical basis in $\mathbb{R}^n$. The elements of $\mathcal{V}$ are given by diagonal matrices, as one can easily obtain from $\varphi(i)\varphi(i)^{\top} = {\rm diag}(e_i)$.  Moreover, the representation of the constant function is achieved by $U = \mathbb{I}_n$. Exploiting the moment point of view, one can easily check that the affine hull $\hat{\mathcal{K}}$ is equivalent to the convex hull $\mathcal{K}$, hence the SOS relaxation is tight. Proving that the approximation is tight using the primal formulation is less obvious. Let $K \in \mathbb{R}^{|\X| \times |\X|}$ the Gram matrix of all dot-products. If $K$ is invertible, then there exists $A \in \mathbb{R}^{|\X| \times d}$ such that $A \varphi$ sends all points to an element of the canonical basis of $\mathbb{R}^{|\X|}$.
We can use this property to expand $f: \X \to \mathbb{R}$ as a SOS: $\ds h(x) = \sum_{x' \in \X} 1_{x'=x} h(x') =  \sum_{x' \in \X} \big[\varphi(x)^\ast A^\ast A \varphi(x') \big]^2 h(x')$.
% Option 2: after projection $  \varphi(x) \varphi(x)^\top $ is diagonal with only a single one in the diagonal, and the simplex is the convex hull of the canonical basis.

\item
\textbf{Quadratic functions in $\X = \mathbb{R}^{d}$:}
Let $\ds \varphi(x) = {1 \choose x} \in \mathbb{R}^{d+1}$ and consider  $\ds h(x) = \frac{1}{2}{1 \choose x} ^\top \bigg( \begin{array}{ll}\! \! c & b^\top\!\! \\ \!\! b & A\!\!
\end{array} \bigg)  {1 \choose x} $. The goal is to understand if $h$ non-negative is equivalent to be a SOS, we use the primal point of view. One can easily show that we need $A \succcurlyeq 0$  otherwise $\inf_{x \in \mathbb{R}^d}  h(x) = - \infty$. By differentiating $h$ one obtains: $\inf_{x \in \mathbb{R}^d}  h(x) = \frac{1}{2}(c-b^{\top}A^{-1}b)$. Hence $h$ is non-negative if the Schur complement is PSD. It follows from simple algebraic fact that $A$ being PSD jointly with the Schur complement being PSD guarantees the condition $\bigg( \begin{array}{ll}\! \! c & b^\top\!\! \\ \!\! b & A\!\!
\end{array} \bigg) \succcurlyeq 0$, therefore using Prop.~1 we can rewrite $h$ as SOS.
% $\bigg( \begin{array}{ll}\! \! c & b^\top\!\! \\ \!\! b & A\!\!
% \end{array} \bigg) \succcurlyeq 0$.

% Extends to the Euclidean sphere.

\item \textbf{Polynomials in $\mathbb{R}$.}
\label{poly1d_tightness} A polynomial $p$ of even degree in $\mathbb{R}$ is non-negative if and only if it is a sum of squares. One can show that by factor $p$ in term of its roots:
$$
p(x) = \alpha \prod_{i \in I} ( x - r_i)^{ m_j} \prod_{j \in J} \big[ (x-a_j)^2 + b_j^2 \big],
$$
it follows easily indeed that $\alpha$ must be positive, the multiplicity of the real roots $\{r_i\}_{i \in I}$ must be even (otherwise we would change sign once crossing them), and by expanding the complex-conjugate roots term we obtain a sum of $2^{|J|}$ squares.

\item \textbf{Order $r$ polynomials in dimension $d$:} All non-negative polynomials are SOS when $r=2$, or when $d=1$. Only other case is $d=2$, and $r=4$ (Hilbert, 1888).

A counter-example is given by Motzkin (1965) for $(d=2,r=6)$: $1+x_1^2 x_2^4 + x_1^4 x_2^2 - 3 x_1^2 x_2^2$.

\item
\textbf{Trigonometric polynomials on $[-1,1]$.}
We consider $\X = [-1,1]$ and $\varphi(x) \in \mathbb{C}^{2r+1}$, with $[\varphi(x)]_\omega = e^{ i\pi \omega x}$ for $\omega \in \{-r,\dots,r\}$. As we discussed before the outer product of the $\varphi$ is given by: $(\varphi(x)\varphi(x)^\ast)_{\omega \omega'} = e^{ i\pi (\omega - \omega') x}$, and thus $\mathcal{V}$ is the set of Hermitian Toeplitz matrices, and we can take $U = \frac{1}{d} I $. For trigonometric polynomials in one dimension non-negativity is equivalent to being SOS, enabled by Fejer–Riesz theorem \cite{Fejér1916} (see \cite{grenandertoeplitz} p. 20): any PSD Toeplitz matrix is of the form $\hat{p}(\omega-\omega')$ where $p$ is a positive measure.
\end{itemize}
\newpage
\section{Lecture 2}
\subsection{Introduction}
In the last lecture we introduced a method to optimize a function $h(x)$ under the condition that it can be expressed as $h(x)=\varphi^\ast(x)H\varphi(x)$, for a feature map $\varphi:\X\mapsto \mathbb{R}^p$.
Two equivalent relaxations can be made:
\begin{itemize}
    \item\textbf{Primal point of view:}
    \begin{align}
    \inf_{x \in \X} h(x) & = \sup_{c \in \mathbb{R}} \ c \mbox{ such that } \forall x \in \X, h(x) - c \geqslant 0  \\
    & \geqslant  \sup_{c \in \mathbb{R}, \ A \succcurlyeq 0} \ c\  \mbox{ such that } \forall x \in \X,  h(x) - c = \varphi(x)^\ast A \varphi(x),
    \label{eq:sos_view_lec2}
\end{align}
where we relaxed the constraint $h(x) - c \geqslant 0$ to $h(x) - c = \varphi(x)^\ast A \varphi(x)$.
\item\textbf{Dual point of view:}
The dual problem \eqref{eq:dual_prob} gives
\begin{align}
    \inf_{x \in \X} h(x)&= \inf_{\mu\in\mathcal{P}^1 (\mathcal{X})} \int_\X h(x)d\mu(x)=\inf_{\Sigma\in\mathcal{K}}\tr[H\Sigma]\geq \inf_{\Sigma\in\widehat{\mathcal{K}}}\tr[H\Sigma],
\end{align}
with $\mathcal{K}=\text{convex hull}\left(\{\varphi(x)\varphi(x)^\ast, \,x\in\X\}\right)$ and $\widehat{\mathcal{K}} =
\big\{  \Sigma \in \mathcal{V}, \ \tr[ U\Sigma ] =1, \ \Sigma \succcurlyeq 0 \big\}$.
\end{itemize}
This relaxation scheme is however affected by some limitations. First we need $h$ to be representable as $h(x)=\varphi^\ast(x)H\varphi(x)$. Moreover the inequalities are tight only for some classes of functions (e.g., polynomials in one dimension, trigonometric polynomials in $[-1,1]$). One way to fix this is to use hierarchies of feature spaces, i.e., using progressively more expressive feature spaces to approach tightness.
To demonstrate this method we start with a couple of examples.
\begin{itemize}
\item
\textbf{Trigonometric polynomials on $[-1,1]^d$}
Suppose $h(x)=\varphi(x)^\ast H\varphi(x)$ on $\X=[-1,1]^d$ with $\varphi(x)=e^{i\omega^T x}$ for $\omega\in\Omega\subset \mathbb{Z}^d$. $\Omega$ basically acts as a constraint on the spectrum of $h$. For this choice of $\varphi$ we have $$\mathcal{V}=\text{Span}\left(\{\varphi(x)\varphi(x)^\ast,\;x\in\X\}\right)=\text{Span}\left(\left\{B(x)\in \mathbb{C}^{|\Omega|\times|\Omega|}\,:\,B_{\omega\omega'}(x)=e^{i(\omega-\omega')^Tx},\; x\in\X\right\} \right).$$ This shows that $\mathcal{V}$ is defined by a set of linear constraints in the space of $|\Omega|\times|\Omega|$ Hermitian matrices. The relaxation \eqref{eq:sos_view} is not tight in general, however by embedding $\Omega$ in a larger set we can make the relaxation as tight as desired, at the price of using larger and larger embeddings.
Consider for example the hierarchy of sets $\Theta_r=\{\omega \in\mathbb{Z}^d\,:\, \norm{\omega}_\infty\leq r\}$. We define new features $\psi_r:\X \mapsto \mathbb{R}^{|\Theta_r|}$, analogously to $\varphi$. Suppose that for sufficiently large $r$, $\Omega\subseteq\Theta_r$. Then we can write $h$ in terms of the new features, i.e., $h(x)=\psi_r(x)^\ast H'\psi_r(x)$, with $H'$ a block matrix of the form
\begin{equation}
    H'=\left(\begin{array}{@{}c|c@{}}
  \begin{matrix}
 H
  \end{matrix}
  & 0\\
\hline
  0 & 0
\end{array}\right)
\end{equation}
in a basis in which the first $|\Omega|$ coordinates are the frequencies in $\Omega$ and the remaining $|\Theta_r|-|\Omega|$ are the frequencies in $\Theta_r/ \Omega$.

Increasing $r$ the relaxation of $h(x)-c\geq 0 $ to $h(x)-c=\psi_r(x)^\ast A\psi_r(x)$ becomes tighter since the feature space becomes more expressive. What is the price to pay? Increasing the number of features means that the dimensionality of the semidefinite problem \eqref{eq:sos_view} increases as well, so it becomes more expensive to solve computationally. In the current example $|\Theta_r|=r^d$.

\item \textbf{Boolean hypercube} Let $\X=\{-1,1\}^d$. Consider the set indexed features $\varphi_A(x)=\prod_{i\in A}x_i$.
Recall that each function $f:\X\mapsto \mathbb{R}$ can be uniquely written as $f(x)=\sum_{A\subseteq[d]} \hat f(A) \varphi_A(x)$; this is just the Fourier transform for Boolean functions. We call $|A|$ the order of $\varphi_A$. Notice moreover that $\varphi_A(x)\varphi_B(x)=\varphi_{A\triangle B}(x)$, where $\triangle$ is the symmetric difference operator.

Consider $\mathcal{A}\subseteq \mathfrak{P}([d])$\footnote{we indicate with $\mathfrak{P}(\cdot)$ the power set.}, a collection of sets in $[d]$. Then $\mathcal{V}=\text{Span}\left(\{\varphi(x)\varphi(x)^\ast,\;x\in\X\}\right)=\text{Span}(\{M(x)\,:\;$ $M_{AB}=\varphi_{A\triangle B}(x),\; x\in\X \})$, which is defined by a set of linear constraints in the space of $|\mathcal{A}|\times|\mathcal{A}|$ matrices. One can then introduce hierarchies based of the order of the features, for example consider the set $\Theta_r=\left\{A\in\mathfrak{P}([d])\;:\; |A|<r\right\}$. Increasing $r$ the relaxation becomes more and more tight.
\item \textbf{General polynomial case (Putinar's Positivstellensatz \cite{putinar1993positive}):}
\begin{theorem}
Consider $K =\{ x \in \mathbb{R}^d\;:\;  \forall j \in [m] , g_j(x) \geq 0\}$ for $g_j:\mathbb{R}^d\mapsto \mathbb{R}$ some multivariate polynomials, and assume that for at least one index $i\in[m]$ the set $\{x\in\mathbb R^d\,:\,g_i(x)\geq0\}$ is compact. \textbf{Then}, if a polynomial $f$ is strictly positive on $K$, there exist sum-of-squares polynomials $\{f_j\}_{j=0}^m$, such that
\begin{equation}
\label{eq:positivstell}
f = f_0 + \sum_{j=1}^m f_j g_j.
\end{equation}
\end{theorem}

Previous results tell us that not every positive polynomial is SOS; instead \eqref{eq:positivstell} is the next best representation that we can obtain, involving multiple sum of squares.
Notice that the degrees of the SOS polynomials are not known in advance, hence hierarchies are needed. In practice one considers the representations $f=f_0+\sum_j g_j f_j$, with deg($f_j$)$<D$ for all $j$. Increasing $D$ the approximation becomes better and better.
%\textcolor{orange}{Anche se lo avevi gia menzionato nei casi prima, magari un commento sul fatto che anche qui aumentando $m$ (che non e' noto a priori) si possa avere approssimazione migliore o e' banale?}
It is also possible to have a moment view of this result.
\item\textbf{Max-cut \cite{goemans1995improved}:}\\
Consider a weighted graph $G=(V,w)$, with $d=|V|$ vertices and $w\in\mathbb{R}_+^{d\times d}$, positive weights.
A cut is a partition of the vertices of $G$ in two groups.
Each partition can be indicated by a vector $x\in\{-1,+1\}^d$ assigning each vertex to either the group $+1$ or $-1$. In max cut the objective is to find the partition such that the sum of the cross group weight is maximized. In terms of $x$ the objective is
\begin{equation}
    C(x)=\frac{1}{2}\sum_{i,j=1}^d w_{ij}\left(1-x_ix_j\right)=\frac{1}{2}\sum_{i,j=1}^d w_{ij}-\frac{1}{2}x^TWx=\frac{1}{2}\sum_{i,j=1}^d w_{ij}-\frac{1}{2}\tr[W xx^T].
\end{equation}
The optimization can then be expressed as
\begin{equation}
    \text{OPT}=\max_x C(x)=\frac{1}{2}\sum_{i,j=1}^d w_{ij}-\frac{1}{2}\min_{\substack{X: \text{rank}(X)=1, \\X\succcurlyeq 0, X_{ii}=1\, \forall i\in[d]} }\tr[W X],
\end{equation}
where $X\in \mathbb R^{d\times d}$, and the constraint $X_{ii}=1\, \forall i\in[d]$ enforces that $x$ is on the hypercube.
The SDP relaxation consists of removing the rank 1 constraint giving the problem
\begin{equation}
\label{eq:SDP_maxcut}
    \text{SDP}= \frac{1}{2}\sum_{i,j=1}^d w_{ij}-\frac{1}{2}\min_{\substack{X:\,X\succcurlyeq 0, X_{ii}=1\, \forall i\in[d]} }\tr[W X].
\end{equation}
\end{itemize}
We have of course $\text{SDP}\geq\text{OPT}$.
At this point we can ask two questions:
\begin{itemize}
\item How loose is the SDP approximation?
\item Given $\hat X$, a solution of \eqref{eq:SDP_maxcut}, how do we obtain a proper cut $x$?
\end{itemize}
To answer both question we devise the following procedure:
Sample $\Tilde{x}\sim\mathcal{N}(0,X)$, estimate the max weight cut as $x=\sign(\Tilde{x})$. 
\begin{proof}
Denote $\EX$ the expectation with respect to $\Tilde{x}$, and write $X_{ij}=u_i^Tu_j$ with $u_i\in\mathbb R^d$. Then we have $x_i=\sign(u_i^Tz)$ with $z\sim\mathcal{N}(0,\mathbb{I}_d)$.
\begin{align}
    \text{OPT}&\geq\EX\left[\frac{1}{2}\sum_{i,j=1}^d w_{ij}\left(1-x_ix_j\right)\right]=\sum_{ij}w_{ij}\mathbb{P}(\Tilde{x}_i\Tilde{x}_j>0)=\\&=\sum_{ij}w_{ij}\mathbb{P}(u_i^T z u_j^Tz>0)=\sum_{ij}w_{ij}\left(1-\frac{1}{\pi}\arccos{(u_i^Tu_j)}\right)=\\&=\sum_{ij}w_{ij}\frac{1}{2} \left(\frac{2}{\pi}\frac{\arccos{(u_i^Tu_j)}}{1-u_i^T u_j}\right) \left(1-u_i^T u_j\right)\overset{(a)}{\geq} \alpha\frac{1}{2}\sum_{ij}w_{ij}(1-X_{ij})=\alpha \text{SDP},
\end{align}
with $\alpha = \frac{2}{\pi} \min_{z \in [-1,1]}  \frac{\arccos(z)}{1-z}\approx 0.87856$. $(a)$ follows from the fact that the probability of $z$ having positive or negative scalar product with both $u_i$ and $u_j$ is $1-\theta_{ij}/\pi$ with $\theta_{ij}$ being the angle between $u_i$ and $u_j$. Figure \ref{fig:maxcut} further explains this point.
\begin{figure}
    \centering
    \includegraphics[width=4.5cm]{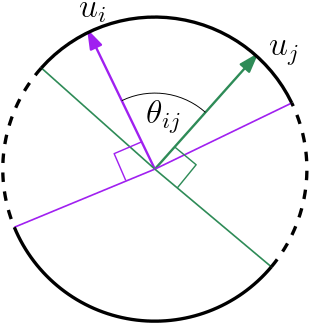}
    \caption{Geometrical representation showing the probability that $(z^T u_i)(z^T u_j)>0$: The continuous black arcs represent where $(z^T u_i)(z^T u_j)>0$, the dashed lines represent where instead $(z^T u_i)(z^T u_j)<0$. The probability that a random angle lands in the continuous sector is $1-\frac{\theta}{\pi}=1-\frac{1}{\pi}\arccos(u_i^Tu_j)$.}
    \label{fig:maxcut}
\end{figure}
\end{proof}
\subsection{Kernel methods}
One way to increase the expressivity of the feature space is to make it infinite dimensional. This is possible through kernel methods.
Consider a map $\varphi:\X\mapsto \mathcal{F}$, with $\mathcal{F}$ some Hilbert space, with inner product $\langle,\rangle$. We can then define a new Hilbert space whose elements are the functions which are linear in $\varphi$. Let $\mathcal{F}'=\left\{g:\X\mapsto \mathbb{R}\,:\, g(x)=\langle f,\varphi(x)\rangle \text{ for some } f\in\mathcal{F}\right\}$. In the following we will assume that $\mathcal F $ is a reproducing kernel Hilbert space (RKHS).
In this case, assuming $k(x,y)=\langle\varphi(x),\varphi(y)\rangle$ is the reproducing kernel, we can identify $\mathcal{F}'$ and  $\mathcal{F}$. In fact by the reproducing property   every function in $\mathcal{F}$ can be expressed in the form $f(x)=\langle f,\varphi(x)\rangle$.
When minimizing functions over $\mathcal{F}$ we have the following result.
\begin{theorem}\label{rep_thm}
\textbf{Representer theorem \cite{kimeldorf1970correspondence}}: \\Let $\mathcal{F}$ be an RKHS, with reproducing kernel $k(x,y)=\langle\varphi(x),\varphi(y)\rangle$. \\ Also let $L:\mathbb{R}^n\mapsto \mathbb{R}$ be an arbitrary error function and $\{x_1,x_2,\dots,x_n\}, \,x_i\in\X $ a set of points. \\Consider the problem
\begin{equation}
    \label{eq:min_hilbert}
    \min_{f\in\mathcal{F}} L\left(\langle f, \varphi(x_1)\rangle,\dots,\langle f, \varphi(x_n)\rangle\right)+g(\norm{f})
\end{equation}
with $g$ an increasing function.\\
\textbf{Then} any minimizer $\hat f$ of \eqref{eq:min_hilbert} can be written as
$f=\sum_{i=1}^n \alpha_i \varphi(x_i)$. Making the function explicit gives
\begin{equation}
    \hat{f}(x)=\sum_{i=1}^n \alpha_i \langle \varphi(x_i),\varphi(x)\rangle=\sum_{i=1}^n \alpha_i k(x,x_i).
\end{equation}
\end{theorem}

In the setting of \ref{rep_thm} we have
\begin{equation} \label{eq:kernel_norm}
    \norm{\hat f}^2=\sum_{ij=1}^n\alpha_i\alpha_j\langle k(\cdot,x_i),k(\cdot,x_j)\rangle=\sum_{ij=1}^n\alpha_i\alpha_j k(x_i,x_j)=\alpha^T K\alpha,
\end{equation}
where we introduced the kernel matrix $K\in\mathbb{R}^{n\times n}$, $K_{ij}=k(x_i,x_j)$.\\
The importance of the representer theorem lies in the fact that the possibly infinite dimensional problem \eqref{eq:min_hilbert} admits a solution in an $n-$dimensional linear function space (moreover for particular forms of  $L,g$ the solution is explicit).

We have seen how an RKHS gives a reproducing kernel $K$. In the other direction we can prove that given a positive semidefinite kernel $k:\X\times \X\mapsto \mathbb{R}$, there exists a unique associated RKHS 
\begin{theorem}
\textbf{Moore Aronszjan \cite{aronszajn1950theory}:}\\
Let $k:\X\times \X\mapsto \mathbb{R}$ be a positive semidefinite kernel, in the sense that for every $\{x_1,\dots,x_n\}$, the kernel matrix $K_{ij}=k(x_i,x_j)$ is positive semidefinite. \\
\textbf{Then} there exists a Hilbert space with inner product $\langle,\rangle$, and a feature map $\varphi: \X\mapsto \mathcal{F}$, such that $k(x,y)=\langle\varphi(x),\varphi(y)\rangle$.
\end{theorem}
One example of RKHS are Sobolev spaces in $\mathbb R^d$.
We define the Sobolev space $\mathcal H_s$ of order $s$ as the space of functions whose derivatives up to order $s$ are in $L^2(\mathbb{R}^d)$. If $s>d/2$, $\mathcal H_s$ can be represented as an RKHS, with kernel $k(x,y)=q(x-y)$. And
\begin{equation}
\label{eq:4ier_kernel_sobolev}
    \hat{q}(\omega)\propto \frac{1}{\left(1+\norm{\omega}_2^2\right)^s},
\end{equation}
with $\hat{q}$ being the Fourier transform of $q$.
The RKHS norm can then be written as
\begin{equation}
    \norm{f}^2_q=\int \, |\hat f(\omega)|^2 \left(1+\norm{\omega}_2^2\right)^s d\omega,
\end{equation}
The larger $s$ the more high frequencies are penalized.
For $s=d/2+1/2$ the kernel corresponding to $q$ is the Laplace kernel (see \cite{bach2021learning} p. 165)
\begin{equation}
    k(x,y)=e^{-\norm{x-y}_2}.
\end{equation}
Another example of RKHS is that induced by the Gaussian kernel $k(x,y)=e^{-||x-y||_2^2}$
\subsection{k-SOS relaxation \cite{marteau2020non}\cite{marteau2022second}\cite{rudi2020finding}}
Consider the SOS relaxation with $\varphi:\X\mapsto \mathcal{F}$, with $\mathcal{F}$ the feature space associated to a Sobolev kernel of order $r$ on $\mathbb R^d$.  We restrict the optimization to $\X\subset \mathbb{R}^d$, e.g., $\X=[-1,1]^d$.
Analogously to \eqref{eq:sos_primal_relaxed} we have
\begin{align}
\inf_{x \in \X} h(x) & = \sup_{c \in \mathbb{R}} \ c \mbox{ such that } \forall x \in \X, h(x) - c \geqslant 0  \\
\label{eq:sos_primal_relaxed_ker}
& \geqslant  \sup_{c \in \mathbb{R}, \ A \succcurlyeq 0} \ c\  \mbox{ such that } \forall x \in \X,  h(x) - c = \varphi(x)^\ast A \varphi(x),
\end{align}
where this time $A:\mathcal{F}\mapsto\mathcal{F}$ is a positive definite self-adjoint operator.
Several questions arise:
\begin{itemize}
\item Is the problem feasible? In other words, can $h(x)-c$ be represented as $\varphi(x)^\ast A \varphi(x)$?  Yes if $h$ is sufficiently differentiable.
\item Is the relaxation tight? Yes, if the minimizer satisfies a local Hessian condition.
\item Is the problem solvable in polynomial time? Recall that in this case the optimization over $A$ would be infinite dimensional so we cannot solve it directly like in \eqref{eq:sos_view}. We will see that we can still solve the problem in polynomial time, through sampling, regularization and a new representer theorem.
\end{itemize}
\subsection{Representation of a non-negative function as a sum-of-squares}
To answer the first question suppose $h:\mathbb{R}^d\mapsto [0,\infty)$ is a positive function in $C^m$ (i.e., $h$ has continuous derivatives up to order $m$). Can we represent $h$ as a sum of squared functions, each of which is in $C^{m'}$, with $m'\leq m$?

One (naive) solution would be to take $h(x)=(\sqrt{h(x)})^2$. This works with $m'=m$ as long as $h$ is strictly positive, in order to preserve differentiability. If $h=0$ somewhere then square root corresponds to $m'=0$. In principle we would like to find solutions in smaller classes of functions (ideally $m'=m$) so the square root is not satisfactory.

We will now state a solution which achieves $m'=m-2$, under some extra conditions
\begin{prop}
\label{prop_SOS}
Let $h:\X\mapsto [0,\infty)$, be a $C^m$ ($m\geq2$) function with $\X\subset\mathbb{R}^d$ a compact set. Also suppose $\hat x$ is the unique global optimum of $h$, located in the interior of $\X$, and with $h(\hat x)=0$. Finally suppose that the Hessian of $h$ at $\hat x$ is positive definite.\\
\textbf{Then} $h$ admits an SOS representation with $m'=m-2$ and at most $d+1$ factors.

\end{prop}
\begin{proof}
Without loss of generality assume $\hat{x}=0$.
We can write $h$'s Taylor expansion in $\hat x$ with integral remainder as
\begin{equation}
    h(x)=h(0)+\nabla h(0)^Tx+\int_0^1(1-t)x^T\nabla^2 h(tx) \,x\, dt.
\end{equation}
    The Hessian must be strictly positive definite near $0$, since we assumed this is an interior point. So we have $\nabla^2 h(tx)\succcurlyeq \lambda I$ for $\norm{x}$ small enough. Using the fact that in zero the gradient and the function vanish we have
    \begin{equation}
        h(x)=x^TR(x)x,\,\, \text{with}\,\,\, R(x)=\int_0^1(1-t)\nabla^2h(tx)dt\succcurlyeq \frac{\lambda}{2}I.
    \end{equation}
    $h$ can then be written in the form
    \begin{equation}
    g(x)=x^T R(x)^{1/2}R(x)^{1/2}x=\sum_{i=1}^d\left(R^{1/2}(x)x\right)^2,
\end{equation}
 where the matrix square root is $C^\infty$ because $R(x)\succcurlyeq \frac{\lambda}{2}I$\cite{del2018taylor}.
with factors that are $C^{m-2}$, i.e., the regularity of the Hessian.
    When instead $\norm{x}$ is not sufficiently small, by the assumptions of unique global optimum and positive definite Hessian we have that $h(x)-h(0)$ is positive, hence $\sqrt{h(x)}$ is $C^m$. Using a partition of unity one can bridge the two expressions, one close and the other far from zero. This yields an SOS representation with $d+1$ terms.
\end{proof}
This result can also be generalized to manifolds of global minima.
As a consequence of proposition \ref{prop_SOS}, if $h$ is a $C^m$ function with isolated global minima, and we pick $\mathcal{F}$ to be the RKHS with kernel \eqref{eq:4ier_kernel_sobolev} and $s \geq m-2$, then the relaxation \eqref{eq:sos_primal_relaxed_ker} is tight. In other words there exists $A^\star$ such that $h(x)-\inf_{x\in\X}h(x)=\varphi(x)^*A^\star\varphi(x)$ for all $x\in\X$.

\subsection{Controlled approximation through subsampling}
The relaxation 
\begin{equation}
\inf_{x \in \X} h(x)  \geqslant  \sup_{c \in \mathbb{R}, \ A \succcurlyeq 0} \ c\  \mbox{ such that } \forall x \in \X,  h(x) - c = \varphi(x)^\ast A \varphi(x),
\end{equation}
 with $A:\mathcal{F}\mapsto\mathcal{F}$ involves an infinite number of constraints (one for every $x\in\X$) and an optimization over operators in infinite dimensions. These facts make the problem impossible to be solved on a computer as is.
 To circumvent the difficulties we use the subsampling technique. In other words we replace $\forall x \in \X,  h(x) - c = \varphi(x)^\ast A \varphi(x)$ with $\forall i\in[n],\,\, h(x_i) - c = \varphi(x_i)^\ast A \varphi(x_i)$,
 where $x_i$ are sampled uniformly in $\X$.
 
 To avoid overfitting we regularize using $\tr[A]$, hence penalizing operators with large trace. Any other function providing control over the eigenvalues would also work.
 The new problem is 
 \begin{equation}
 \label{eq:reg_subsamp_primal}
     \sup_{c \in \mathbb{R}, \ A \succcurlyeq 0} c-\lambda\tr[A] \text{ s.t. }\forall i\in[n],\,\, h(x_i) - c = \varphi(x_i)^\ast A \varphi(x_i).
 \end{equation}
 Notice that the solution to this problem is a lower bound to the right hand side of 
 \begin{prop}
 In the setting \eqref{eq:reg_subsamp_primal} let $h$ be $C^m$, $\lambda>0$, $A:\mathcal{F}\mapsto\mathcal{F}$, with $\mathcal F$ a Sobolev space of order $s\leq m-2$. Let $n\geq C(n,d) \epsilon^{-\frac{d}{m-3}}$. \textbf{Then} with high probability with respect to sampling
 \begin{equation}
     \left|\left[\inf_{x \in \X} h(x)\right]-\left[\sup_{c \in \mathbb{R}, \ A \succcurlyeq 0} c-\lambda\tr[A] \text{ s.t. }\forall i\in[n],\,\, h(x_i) - c = \varphi(x_i)^\ast A \varphi(x_i)\right]\right|<\epsilon.
 \end{equation}
 \end{prop}
 The statistical lower bound for the number of queries is $n\geq C'(d)\epsilon^{-\frac{d}{m}}$, however exponential computation (in $n$) might be required to achieve it. Notice that setting $m=0$ we recover the curse of dimensionality $n\geq \epsilon^{-d}$. %In the SOS formulation the problem will be polynomial in $n$, at the price of losing three orders of derivability (i.e. $d/m\mapsto d/(m-3)$). \bl{GP:check!}
 \\
 The regularization is necessary to exploit the differentiable structure of $h$, in fact setting $\lambda=0$ one would get $c^\star=\min_i h(x_i)$, requiring $\epsilon^{-d}$ samples, regardless of $m$.
 One drawback  of the subsampling approach is the fact that $C$ can depend exponentially on $d$.
 
 The previous proposition solves the problem of the infinite number of constraints, but we're still dealing with an optimization problem over the space of infinite dimensional operators. The following representer theorem tells us that we can equivalently solve a finite dimensional problem
 \begin{theorem}
 (Representer theorem for SOS) For any $x_1,\dots,x_n\in \mathcal X$ any solution of \eqref{eq:reg_subsamp_primal} is of the form 
 \begin{equation}
     A=\sum_{i,j=1}^n\varphi(x_i)\varphi(x_j)B_{ij},\quad B\in\mathbb{R}^{n\times n}, B \succcurlyeq 0.
 \end{equation}
 
 \end{theorem}
 For a proof see \cite{marteau2020non}.
 The theorem can be generalized to the case where the regularizer (here $\tr[A]$) is an arbitrary spectral norm.
 Letting $K$ be the kernel matrix this immediately implies 
 \begin{equation}
     \varphi(x_i)^\ast A \varphi(x_i)=\sum_{l,j}K(x_i,x_j)B_{lj} K(x_j,x_i)=(KBK)_{ii},
 \end{equation}
 So that the final problem becomes
 \begin{equation}
 \label{eq:final_subsamp}
      \sup_{c \in \mathbb{R}, \ B \succcurlyeq 0} c-\lambda\tr[KBK] \text{ s.t. }\forall i\in[n],\,\, h(x_i) - c = (KBK)_{ii},
 \end{equation}
 which is now an $n$ dimensional semidefinite programming problem, solvable in $O(n^{3.5})$.
 
 From a practical standpoint \eqref{eq:final_subsamp} is solvable up to $n\approx 10^4$.
 Moreover it is crucial that the kernel matrix $K$ is invertible.

\newpage
\section{Lecture 3: From optimization to information theory}
\subsection{Introduction}
Recall that our goal is to take the $\min_{x \in \cX} h(x) $, where $\cX$ is a generic set of inputs. To this extent, we reformulated the minimization problem using the sum-of-square relaxation:
\begin{align}
\min_{x \in \cX} h(x) = \sup_{c, A \succcurlyeq 0} c \quad \text{s.t.} \quad h(x) = c+\langle \phi(x),A\phi(x) \rangle   \qquad \forall x \in \cX .\label{eq:sos_intro3}
\end{align}
In other words, we replaced the positivity of a function by being a sum-of-square of some feature vector $\phi: \cX \to \cF$. In particular, we assumed the feature map $\phi$ to be associated with a kernel $\cK$. We showed that altough typically there is a gap between the LHS and RHS in~\eqref{eq:sos_intro3}, with mild assumptions on $h$ there is tightness. Moreover, we replaced the continuous set of equalities in~\eqref{eq:sos_intro3}, by the following: 
\begin{align} \label{eq:discrete_ineq}
\sup_{c, A \succcurlyeq 0} c - \lambda \tr A \quad \text{s.t.} \quad \forall i \in \{ 1,...,n\} \quad h(x_i) = c+\langle \phi(x_i),A\phi(x_i) \rangle .
\end{align}
We explained that to solve~\eqref{eq:discrete_ineq} in practice, one can use the representer theorem (Theorem~\ref{rep_thm}), which allows to write 
$\langle \phi(x_i),A \phi(x_i) \rangle = (KBK)_{ii} $, with $B \in \bR^{n \times n}$ and $B\succcurlyeq 0$ and $K$ kernel matrix, and similarly $\tr A = \tr (BK)$, and therefore one can solve~\eqref{eq:discrete_ineq} in polynomial time.

In this lecure, we move beyond optimization and extend the sum-of-square technique to other problems: e.g., optimal control, approximation of entropies and relative entropies in information theory.

\subsection{Extension to other infinite dimensional problem}
Recall, the two equivalent problems defined in the previous lectures:
\begin{itemize}
\item   {Primal problem}
\begin{align} \label{eq:primal_prob}
\displaystyle \min_{x \in \cX} \ h(x) = \sup_{c \in \mathbb{R}}\  c \ \ \mbox{ such that } \ \ \forall x \in \cX,\  h(x) - c \geqslant 0.
\end{align}

\item  {Dual problem on probability measures}
\begin{align} \label{eq:dual_prob}
\inf_{\mu \in \mathbb{R}^\cX} \int_\cX \mu(x) h(x) dx \ \ \mbox{ such that } \ \int_\cX \mu(x) dx =1, \ \forall x \in \cX,\ \mu(x) \geqslant 0.
\end{align}
\end{itemize}
The two are formulated on positive functions, and we could a priori apply the sum-of-square relaxation to either of the two inequalities. However, note that in~\eqref{eq:dual_prob} we expect $\mu$ to be a Dirac's delta distribution at the optimum, which implies that the sum-of-square relaxation will not work. On the other hand, in~\eqref{eq:primal_prob}, we expect that the objective function $h$ of global optimization problems is well behaved (e.g. smooth), thus the sum-of-square relaxation will work well.

We now give a generic formulation of a constrained optimization problem, and then specialize it to optimal control. The generic problem can be expressed in the following form:
\begin{align} \label{eq:genericproblem}
    \inf_{ \theta \in \Theta}\  F(\theta) \ \ \mbox{ such that } \ \ \forall x \in \cX, \ g(\theta,x) \geqslant 0,
\end{align}
with $F$ convex and $g$ linear in its first argument $\theta$, and where $\Theta$ denotes a generic vector space.
The sum-of-squares reformulation of~\eqref{eq:genericproblem} is given by
\begin{align} \label{eq:genericSOS}
\inf_{ \theta \in \Theta, \ A \succcurlyeq 0} F(\theta) \ \ \mbox{ such that } \ \ \forall x \in \cX, \ g(\theta,x) {}= \langle \phi(x), A \phi(x) \rangle.
\end{align}
Solving~\eqref{eq:genericSOS} requires penalizing the unconstrained problem by $\Tr(A)$ and performing subsampling. If we expect $g(\theta,x)$ to be smooth at the optimum $\theta^*$, the representation as sum-of-squares allows to benefit from its intrinsic smoothness and get sample complexity guarantees. The sum-of-square relaxation can be performed in the primal and in the dual problem, however depending on the smoothness of these problems either the primal or the dual problem might be a better choice. Let us look at the special case of optimal control (or reinforcement learning) as an example. Another relevant example is optimal transport, which we do not cover here and refer to~\cite{vacher2021dimension}.
% The first one is Optimal Transport (see~\cite{vacher2021dimension}), which we do not cover here. The second one is Optimal control or reinforcement learning, which we cover below.
% \paragraph{Example 1: Optimal transport~\cite{vacher2021dimension}} 
% {\color{purple} Not done in class... Decide what to write}

\paragraph*{Example: Optimal control/Reinforcement learning~\cite{berthier2021optimization}.} Consider the dynamical system
\begin{align}
    &\dot X(t) = f(t, X(t), u(t)) \qquad \forall t\in[t_0, T], \\
    &X(t_0) = x_0,
\end{align}
where $X:[t_0,T]\to \cX $ denotes the state, $u:[t_0,T]\to \mathcal{U}$, denotes the control that we want to impose, for some space $\mathcal{U}$ (usually $\mathcal{U} =\cX$) and $t_0,x_0$ are the initial time and the initial state respectively.
In the optimal control formulation~\cite{liberzon2011calculus}, we choose the control $u(.)$ such that it minimizes a cost, i.e., for fixed T  we want to solve
\begin{align} \label{eq:opt_control}
 \displaystyle V^*(t_0, x_0) = \inf_{u: [t_0,T] \to \mathcal{U}}  \int_{t_0}^T L(t, x(t), u(t)) \text{d}t + M(x(T)),
\end{align}
where we can think of the first term as the cost along the way and the second term as the terminal cost. The solution of~\eqref{eq:opt_control} can be found as a subsolution to the PDE defined by the Hamilton-Jacobi-Bellman (HJB) equation~\cite{vinter1993convex}. In particular, any $V$ that satisfies the following
\begin{align}
    & \displaystyle\sup_{V: [0,T] \times \mathcal{X} \to \mathbb{R}} \int V(0, x_0) \text{d}\mu_0(x_0) \nonumber\\
     & \displaystyle \frac{\partial V}{\partial t}(t, x) + L(t, x, u) + \nabla V(t, x)^\top f(t, x, u)\  \geqslant 0\quad \forall (t, x, u), \label{eq:HJB}  \\
        & V(T, x) = M(x)  \quad \forall x ,\nonumber
\end{align}
is a solution at time $t_0=0$ of~\eqref{eq:opt_control}. Among the works that tackled the above equation, we mention~\cite{lasserre2008nonlinear}, that approaches~\eqref{eq:HJB} through sum-of-squares relaxation; and~\cite{berthier2021optimization} for extension to kernel sums-of-squares.

\subsection{Connection to information theory: log partition function~\cite{bach2022sum,bach2022information}}
% {\color{purple} cite: {\small \url{https://arxiv.org/pdf/2202.08545.pdf}} and {\small \url{https://arxiv.org/pdf/2206.13285.pdf}}}
In this section we show how the sum-of-squares relaxation can be applied to compute relevant information theoretic quantities. For simplicity, we will consider the problem of maximizing a function $h(x)$, instead of minimizing it. Specifically, we consider the log partition function~\cite{bach2022information,bach2022sum} of the function, defined as
\begin{align} \label{eq:logsumexp}
\varepsilon \log \int_\X \exp \left( \frac{ h(x)}{\varepsilon}\right) d q(x),
\end{align}
where we call $q(.) \in \mathcal{P}(\X)$ the base measure and $\epsilon>0$ the temperature. There are several reasons that motivate the study of the log partition function: for instance, it is used as a smooth approximation of the maximum, and for computational purposes in optimization; moreover, it is used in probabilistic inference~\cite{wainwright2008graphical}. Our goal is to understand how to compute~\eqref{eq:logsumexp}. Recall, the definition of KL-divergence. 
\begin{definition}[Kullback-Leibler (KL) divergence]
Let $p,q $ be in $\cP^1(\cX)$ (i.e. probability measures on a measurable space $\cX $), such that $p$ is absolutely continuous with respect to $q$. Then, the KL-divergence between p and q is defined as
\begin{align}
    D(p \| q): = \int_{\cX} \log \left( \frac{dp(x)}{dq(x)}  \right) dp(x).
\end{align}
\end{definition}
Note that~\eqref{eq:logsumexp} can be rewritten as
\begin{align}
 \varepsilon \log \int_\X \exp \left( \frac{h(x)}{\varepsilon}  \right) d q(x)
& =  \sup_{p \in \cP^1(\cX)} \int_\X h(x) dp(x) -  \varepsilon\int_\X \log \frac{dp(x) }{dq (x)} dp(x)
\\
& =  \sup_{p \in \cP^1(\cX)} \int_\X h(x) dp(x) -  \varepsilon D(p\|q).  \label{eq:equivlogsumexp}
\end{align}
The first equality is a classical result of convex duality between maximum likelihood and relative entropy. One can verify it by deriving the right-hand side. Let us look at the two terms in~\eqref{eq:equivlogsumexp}. If $h$ is represented as $h(x) := \phi(x)^T H \phi(x) $, then the term $\int_{\cX} h(x) dp(x) $ will be expressed in terms of moment matrices. We therefore ask the following question for the second term: can we upper or lower bound $D(p \| q) $ in terms of the moments of $\phi(x)$, i.e., in terms of 
\begin{align}
    \Sigma_p = \int_{\cX} \phi(x)\phi(x)^Tdp(x) \qquad \text{and} \qquad  \Sigma_q = \int_{\cX} \phi(x) \phi(x)^T dq(x) \quad ? \label{eq:sigmap_sigmaq}
\end{align}
This would lead to a convex optimization problem on the set of moment matrices, which we can relax using sum-of-squares and solve efficiently. A first attempt to answer this question appears in~\cite{bach2002kernel}, and consists of considering $\Sigma_p$ and $\Sigma_q$ to be covariances of Gaussian distributions $p$ and $q$ of dimension $d$. The KL divergence for Gaussian distributions then reads
\begin{align}
- \frac{1}{2} \log \det ( \Sigma_p \Sigma_q^{-1}) + \frac{1}{2}  \tr \Sigma_p \Sigma_q^{-1} - \frac{d}{2}.
\end{align}
This has some nice properties: in fact it is zero if and only if $\Sigma_p = \Sigma_q$, and potentially there is a link with information theory. However, some other properties of $D(p\| q)$ are not preserved: it is no longer jointly convex in $p$ and $q$, it diverges in infinite dimensions and there is no direct link with the true KL divergence. A more powerful approach uses the kernel KL divergence.

\subsection{Kernel KL divergence}
\begin{definition}[Kernel KL divergence, or Von Neumann divergence~\cite{bach2022information}] Let $p$ and $q$ be in $\cP^1(\cX)$, and let $\Sigma_p$ and $\Sigma_q$ be defined as in~\eqref{eq:sigmap_sigmaq} for some feature map $\phi: \cX \to \cF$. We define the kernel KL divergence (or Von Neumann divergence) as
\begin{align}
D(\Sigma_p \| \Sigma_q) = \tr \big[ \Sigma_p (\log \Sigma_p - \log \Sigma_q) \big].
\end{align}
\end{definition}
We will assume that $\phi: \cX \to \cF$ is obtained from a positive definite kernel. We focus mostly on $\bR^d$, however, we remark that one can leverage this technique to any structured objects beyond finite sets and $\bR^d$. The kernel KL divergence has several nice properties, that we list here.
\begin{proposition}[Properties of kernel KL divergence.] \label{prop:propertiesentropy}
$D(\Sigma_p \| \Sigma_q)$ satisfies the following properties:
\begin{enumerate}
\item $D(\Sigma_p \| \Sigma_q)$ is jointly convex in $\Sigma_p$ and $\Sigma_q$. Because $\Sigma_p$ and $\Sigma_q$ are linear in $p$ and $q$, $D(\Sigma_p \| \Sigma_q)$ is jointly convex in $p$ and $q$;
\item $D(\Sigma_p \| \Sigma_q)\geq 0$ and $D(\Sigma_p \| \Sigma_q)= 0$ if $p=q$;
\item If the kernel that generates $\phi$ is universal, then $D(\Sigma_p \| \Sigma_q)= 0$ if and only if $p=q$;
\item If $ p$ is absolutely continuous with respect to $q$, with $ \| \frac{dp}{dq}\|_{\infty} \leq \alpha $, then  $D(\Sigma_p \| \Sigma_q) \leq\log \alpha \cdot \Tr \Sigma_p$.
\end{enumerate}
\end{proposition}
\begin{proof} 
\begin{enumerate}
\item is not trivial, and we refer to the appendices in~\cite{bach2022information} for a formal proof. 
\item holds since $D(\Sigma_p \| \Sigma_q)$ is the Bregman divergence of $\Tr(\Sigma \cdot \log  \Sigma)$ (see e.g.~\cite{amari2016information} for definition and applications of Bregman divergence).
\item follows from the injectivity of the map $p \mapsto \Sigma_p$ when the kernel is universal. In fact, if $\Sigma_p =\Sigma_q$, the universality of the kernel implies that for all continuous functions $f$, $\int_{\cX} f(x) [dp(x) -dq(x)] = 0$, hence $p=q$ and the map $p \mapsto \Sigma_p$ is injective.
\item Assume without loss of generality that $\alpha \geq 1$. We have then $\Sigma_p \preccurlyeq \alpha \Sigma_q $, which leads to
\begin{align}
D(\Sigma_p \|\Sigma_q) =\Tr \big[ \Sigma_p (\log \Sigma_p  \log \Sigma_q) \big] \leq \Tr \big[ \Sigma_p (\log( \alpha\Sigma_q) - \log \Sigma_q) \big] =  \log \alpha \cdot \Tr \Sigma_p.
\end{align}
\end{enumerate}
\end{proof}
\paragraph*{Estimation from data.}
We may ask whether $D(\Sigma_p \| \Sigma_q ) $ can be estimated efficiently from data. To this purpose, assume that $\Sigma_q$ is known and that we observe $x_1,...,x_n$ iid samples from $p$. One can define the empirical moment $\hat \Sigma_p = \frac{1}{n} \sum_{i=1}^n \phi(x_i)\phi(x_i)^T$, and the plug-in estimator $D(\hat \Sigma_p \|\Sigma_q) $. It turns out (Proposition 7 in~\cite{bach2022information}) that with further conditions on the decay of the eigenvalues of the kernel,
\begin{align}
D(\hat \Sigma_p \|\Sigma_q) - D(\Sigma_p \|\Sigma_q) = O\left(\frac{\log n}{\sqrt{n}}\right ).
\end{align}
%\bl{GP: Io nelle note avevo $<1/\sqrt{n}$, però nel paper guardano $D(\hat \Sigma_p \|\hat \Sigma_q) - D(\Sigma_p \|\Sigma_q)$ e gli viene il $\log$, non so quale sia giusto.} {\color{purple} EC: anch'io nelle note ho $1/\sqrt{n}$, ma a me sembra che il log ci sia. Io penso che scriva $O(1/\sqrt{n})$ ignorando eventuali log(n), come ho gia visto in altri paper. Alla fine per la complessita ti interessa l'esponente di n. Posso scrivere $\tilde O(1/\sqrt{n})$, $\tilde O$ dovrebbe essere la notazione standard che ignora i termini logaritmici.. }
We remark that, surprisingly, no extra regularization to the plug-in estimator is required for achieving the above rate.  \\
\\
It is interesting to compare $D(\Sigma_p \| \Sigma_q)$ with the Maximum Mean Discrepancy (MMD) between $p$ and $q$, defined as follows~\cite{gretton2012kernel}: 
\begin{align}
{\rm MMD}(p,q) : = \| \mu_p - \mu_q\|,
\end{align}
where $\mu_p = \int_{\cX} \phi(x) dp(x) $, and similarly $\mu_q = \int_{\cX} \phi(x) dq(x)$, and $\|.\|$ denotes the kernel norm of induced by $\phi(.)$ (see~\eqref{eq:kernel_norm}). Indeed, all properties of $D(\Sigma_p\| \Sigma_q) $ in Proposition~\ref{prop:propertiesentropy} and estimation in $O(\log n/\sqrt{n})$ rate hold for MMD as well. The main difference between $D(\Sigma_p \| \Sigma_q) $ and MMD, is that MMD has no link with the classical notions of divergences from information theory, whereas $D(\Sigma_p \| \Sigma_q) $ has a direct link with the KL divergence, which we analyse in the following Section.

\subsection{Link with KL divergence} \label{sec:linkwithentropy}
If we assume that $\cX$ is a finite set with orthonormal embeddings (i.e. such that $\langle \phi(x), \phi(y) \rangle = 1_{x=y}$) and that all covariance operators are jointly diagonalizable with probability mass values as eigenvalues, then we recover the KL divergence \emph{exactly}:
\begin{align}
D(\Sigma_p \| \Sigma_q) & = \Tr \left[ \Sigma_p \left( \log\Sigma_p - \log 
 \Sigma_q  \right) \right] \\
 & \overset{(a)}{=} \sum_{x \in \cX }  p(x) \left( \log p(x) - \log q(x) \right) = D(p \|q),
% & = D \Big( \sum_{x \in \cX} \phi(x)  \phi(x)^T p(x) \Big\|
% \sum_\X \frac{q(x)}{p(x)}\phi(x)   \phi(x)^T p(x)\Big) \\
% &= \sum_{x \in \X} p(x) \log \frac{p(x)}{q(x)} = D(p\|q).
\end{align}
where $(a)$ holds because by assumption both $\Sigma_p$ and $\Sigma_q$ are diagonal, with the probability mass values as diagonal elements. 

However, beyond finite sets we cannot recover the KL divergence exactly. If we assume that for all $x \in \cX$, $\| \phi(x) \| \leq 1$, then 
% The kernel relative entropy can then be rewritten as
% If we further assume that $\cX$ is a finite set with orthonormal embedding, i.e. such that 
% $\langle \phi(x), \phi(y) \rangle = 1_{x=y}$, and that all covariance operators are jointly diagonalizable with probability mass values as eigenvalues, then we recover the regular entropy \emph{exactly}:
% \begin{align}
% D(\Sigma_p \| \Sigma_q) = \sum_{x \in \X} p(x) \log \frac{p(x)}{q(x)} = D(p\|q).
% \end{align}
% On the other hand, beyond finite sets 
we can find the following lower bound on KL divergence:
\begin{align}
D(\Sigma_p \| \Sigma_q)
& =  D \Big( \int_\X \phi(x)  \phi(x)^T dp(x) \Big\|
\int_\X \frac{dq(x)}{dp(x)}\phi(x)   \phi(x)^T dp(x) \Big)\\
& \overset{(a)}{\leq} \int_\X
D \Big(
\phi(x)  \phi(x)^T \Big\|
 \frac{dq(x)}{dp(x)} \phi(x)   \phi(x)^T
\Big)
dp(x)
\\
& \overset{(b)}{=}  \int_\X
\| \phi(x)\|^2 D \Big(1  \Big\|
 \frac{dq(x)}{dp(x)}
\Big)
dp(x) \\
&  \overset{(c)}{\leq} \int_\X \log \Big( \frac{ dp}{dq}(x) \Big) dp(x) =  D(p\|q), \label{eq:entropybound}
\end{align}
where $(a)$ follows by the \emph{joint} convexity of $D(\cdot\|\cdot)$ (note that both integrals in the arguments are in dp(x)) and by Jensen's inequality, (b) holds because $\phi(x)\phi(x)^T$ and $\phi(x)\phi(x)^T \frac{dq(x)}{dp(x)} $ are rank $1$ operators with eigenvectors proportional to $\phi(x)$ and eigenvalues equal to either zero or $\|\phi(x) \|^2$ (for $\phi(x)\phi(x)^T$) and to either zero or $\|\phi(x) \|^2  \frac{dq(x)}{dp(x)}$  (for $\phi(x)\phi(x)^T \frac{dq(x)}{dp(x)}$), and $(c) $ follows by the assumption that $\|\phi(x)\| \leq 1$ and by the definition of KL divergence. 
% Thus, the kernel entropy is a lower bound to the KL divergence, with the only assumption that features are bounded.

We may ask whether the above inequality is tight. One way to estimate the tightness of~\eqref{eq:entropybound} uses quantum measurements. We outline here the main ideas, and refer to Section 4.1 in~\cite{bach2022information} for details.
% In fact, when measuring we reduce the entropy. 

For all $y \in \cX$, a \emph{quantum measurement} for $p$ is defined as
\begin{align}
\tilde p(y) = \Tr \left( \Sigma_p D(y) \right),
\end{align}
where $D(y)$ is called the \emph{measurement operator}, and it is such that $ D(y) \succcurlyeq 0 $ and $ \int_{\cX } D(y) d \tau (y) = I$, with $\tau$ being the uniform measure over $\cX$. One can define the analogous quantum measurement for $q$ as
\begin{align}
\tilde q(y) = \Tr \left( \Sigma_q D(y) \right).
\end{align}
By data processing inequality (see Appendix A.2 of~\cite{bach2022information}), we have 
\begin{align}
D(\tilde p \| \tilde q ) \leq D( \Sigma_p \| \Sigma_q ).
\end{align}
One can then choose $D(.)$ such that $\tilde p $ and $\tilde q$ are smooth versions of $p$ and $q$. Specifically, for all $y \in \cX$ let 
\begin{align}
D(y) = \Sigma^{-1/2} ( \phi(y) \phi(y)^T)  \Sigma^{-1/2},
\end{align}
which satisfies the above properties. We then have 
\begin{align}
\tilde p(y) = \Tr\left[ D(y) \Sigma_p \right] = \int_\cX \langle \phi(x), \Sigma^{-1/2}\phi(y)   \rangle^2 dp(x) = \int_\cX h(x,y) dp(x),
\end{align}
where $h(x,y):= \langle \phi(x), \Sigma^{-1/2}\phi(y)   \rangle^2 $ is such that $\int_\cX h(x,y) d\tau(x) =1$ and can be seen as a smoothing kernel.

Overall, we get the following two sided inequality:
\begin{align}
D( \tilde p \| \tilde q) \leq D( \Sigma_p \| \Sigma_q) \leq D(  p \|  q), 
\end{align}
that can lead to quantitative bounds between $ D( \Sigma_p \| \Sigma_q) $ and $D(  p \| q)$, specifically when the smoothing function $h$ puts most of its mass on pairs $(x,y)$ where $x $ is close to $y$. For instance, one can see that if $h(x,y) = \exp\left( -\frac{ \|x-y\|_2 }{\sigma} \right)$, then $D(p\|q) - D( \Sigma_p \| \Sigma_q) = O(\sigma^2) $ (see Section 4.2 in~\cite{bach2022information}). However, note that the smaller the $\sigma$ is, the larger the sample will need to be to estimate the kernel KL divergence.

Let us now look at a simple example.
\paragraph*{Example.} Let $\cX = [0,1]$ and $\phi(x) = \exp\left( 2i\omega \pi x \right) \hat q(\omega)^{1/2}$, where $\hat q(\omega)$ is the Fourier transform of a kernel. Then, 
\begin{align}
\phi(x) \phi(y)^* = \sum_{\omega \in \bZ} \hat q(\omega) \exp\left( 2 i \omega (x-y) \right) = q(x-y),
\end{align}
that is the translation-invariant kernel on the torus. $\Sigma_p$ is then an infinite dimensional matrix with elements given by
\begin{align}
   \forall \omega,\omega' \in \bZ, \qquad  (\Sigma_p)_{\omega \omega'} &= \int_{x \in \cX} dp(x) \exp \left( 2i \pi x(\omega - \omega') \right) \hat  q(\omega)^{1/2} \hat q(\omega')^{1/2}\\
    & \overset{(a)}{=} \hat p(\omega - \omega')\hat  q(\omega)^{1/2} \hat q(\omega')^{1/2},
\end{align}
where $(a)$ holds by definition of characteristic function of $p$. Thus, 
\begin{align}
\Sigma_p = {\rm diag}(\hat q)^{1/2} {\rm TM}(\hat p) {\rm diag}(\hat q)^{1/2},
\end{align}
where we denoted by $ {\rm TM}(\hat p) $ the Toeplitz matrix whose elements are given by
\begin{align}
    ({\rm TM}(\hat p))_{\omega,\omega'} = \hat p(\omega - \omega').
\end{align}
The previous discussion implies that the quantity $\tr \Sigma_p \log \Sigma_p$ is related to entropy. 

\subsection{Estimation of log-partition function}
Recall, the definition of log-sum-exp entropy and its equivalent formulation in~\eqref{eq:equivlogsumexp}: 
\begin{align}
  \varepsilon  \log \int_\X e^{h(x)/ \varepsilon} dq(x) &= \sup_{ p \in \cP^1(\cX)} \int_\cX h(x) dp(x) - \varepsilon D( p \| q)
\end{align}
Assuming $\|\phi(x)\|\leq 1$ for all $x \in \cX$, the discussion of Section~\ref{sec:linkwithentropy} allows us to upper-bound the above by 
\begin{align} \label{eq:sumexpfirstrelax}
    \varepsilon  \log \int_\X e^{h(x)/ \varepsilon} dq(x) \leq  \sup_{ p \in \cP^1(\cX)} \int_\cX h(x) dp(x) - \varepsilon D( \Sigma_p \| \Sigma_q).
\end{align}
Moreover, if $h$ is representable, i.e. if $h = \langle \phi(x), H \phi(x) \rangle $ for some $H$, then 
\begin{align}
\int_\cX h(x) dp(x) &= \tr \left( H \int_{x\in \cX} p(x) \phi(x) \phi(x)^T \right)\\
& = \tr\left(H \Sigma_p \right),
\end{align}
and we can further relax~\eqref{eq:sumexpfirstrelax}, with standard SOS relaxation seen before, by replacing the $\sup_p$ by the $\sup_{\Sigma_p \in \hat \cK}$, where we recall 
\begin{align}
    \hat \cK = \{ \Sigma \in {\rm span}\{ \phi(x) \phi(x)^T : x \in \cX   \} : \Sigma \succcurlyeq 0, \tr[U\Sigma]=1\}  . 
\end{align}
Overall, we obtain 
\begin{align} \label{eq:entropyfinalbound}
\varepsilon  \log \int_\X e^{h(x)/ \varepsilon} dq(x) \leq  \sup_{\Sigma_p \in \hat \cK} \tr[H \Sigma_p] - \varepsilon D(\Sigma_p \| \Sigma_q).
\end{align}
This can be computed in polynomial time by semi-definite programming. We remark that adding the relative entropy as a regularizer is a standard technique used to solve the SDP problem~\eqref{eq:discrete_ineq}. Here we show that this technique has also an information-theoretic meaning.
% We remark one `fun' property: $D(\Sigma_p \|\Sigma_q) $ is concave in the kernel. This implies that it can be maximized to obtain better bounds.

\subsection{Extensions}
We highlight the following extensions to the contents presented in this lecture.
\begin{enumerate}
    \item For two probability distributions $p$ and $q$ such that $p \succcurlyeq q$, for a convex function $f$ such that $f(x)$ is finite for all $x>0$ and $f(1)=0$, we can define the \emph{f-divergence} as
    \begin{align}
    D_f (p\| q) = \int_{\cX} f\left( \frac{dp(x)}{dq(x)} \right) dq(x).
    \end{align}
    This definition includes several well-known divergences, e.g. KL, Pearson, Hellinger, $\chi^2$. We note that all the properties discussed above apply to any well defined f-divergence. 
    \item $D(\Sigma_p \|\Sigma_q) $ is concave in the kernel, which means that if we replace for instance $\phi(x) $ by $\Lambda^{1/2}\phi(x)$, with $\Lambda \succcurlyeq 0$, and consider the kernel $K_{\Lambda} (x,y) = \phi(x)^T \Lambda \phi(y)$, then $D(\Sigma_p \|\Sigma_q) $ is concave in $\Lambda$. This property implies that the kernel can be optimized to obtain better bounds.
    \item One can use other notions of quantum divergences, which lead to better bounds. E.g.,~\cite{matsumoto2015new} showed that 
    \begin{align}
    \tr\left[ A ( \log A - \log B)\right]\leq  \tr \left[ A \log(B^{-1/2}A B^{-1/2})\right],
    \end{align}
    which implies that using the right hand side as divergence gives an improvement in the bounds in~\eqref{eq:entropybound} and~\eqref{eq:entropyfinalbound}.
    \item We showed that $D(p\|q) \geq D(A\|B)$ if $A = \Sigma_p$ and $B=\Sigma_q$. We may ask what is the best lower bound on $D(p\|q) $ based on $\Sigma_p$ and $\Sigma_q$:
    \begin{align}
        D^*(A \| B):= \inf_{p, q \in \cP^1(\cX) }\  D(p\|q) \ \mbox{ such that } \ \Sigma_p = A  \mbox{ and } \Sigma_q = B.
    \end{align}
    Indeed, $D^*(A\|B)$ is computable by sum-of-squares relaxation (see~\cite{bach2022sum} for details).
\end{enumerate}
\section{Acknowledgements}
These are notes from the lecture of Prof. Francis Bach given at the summer school \textquotedblleft Statistical Physics \& Machine Learning\textquotedblright, that took place in Les Houches School of Physics in France from 4th to 29th July 2022. The school was organized by Prof. Florent Krzakala and Prof. Lenka Zdeborová from EPFL.
\newpage

\nocite{*}

\bibliography{apssamp}% Produces the bibliography via BibTeX.

\end{document}